\documentclass[11pt,a4paper,twoside]{article}

\usepackage{fullpage}
\usepackage{xcolor,bm, bbm}
\usepackage[utf8x]{inputenc}
\usepackage{array}
\usepackage{mathrsfs}
\usepackage{mathtools}

\usepackage{fourier}
\usepackage{latexsym,amsfonts,amsmath, amsthm, amssymb}

\newcommand{\R}{\mathbb R}
\newcommand{\N}{\mathbb N}

\newcommand{\E}{\mathbb E}
\newcommand{\Pro}{\mathbb P}

\newcommand{\SSS}{\ensuremath{{\mathbb S}}}

\DeclareMathOperator{\id}{id}

\newtheorem{thm}{Theorem}[section]

\newtheorem{lemma}[thm]{Lemma}

\newtheorem{proposition}[thm]{Proposition}

\newtheorem{thmalpha}{Theorem}

\newtheorem{rmkalpha}{Remark}

{
\theoremstyle{definition}

}

\allowdisplaybreaks

\begin{document}

\title{\bf The minimal spherical dispersion}

\medskip

\author{ 
Joscha Prochno and Daniel Rudolf}



\date{}

\maketitle

\begin{abstract}
\small
We prove upper and lower bounds on the minimal spherical dispersion, improving upon previous estimates obtained by Rote and Tichy [Spherical dispersion with an application to polygonal approximation of curves, Anz. \"Osterreich. Akad. Wiss. Math.-Natur. Kl. 132 (1995), 3--10].  In particular, we see that the inverse $N(\varepsilon,d)$ of the minimal spherical dispersion is, for fixed $\varepsilon>0$,  
linear in the dimension $d$ of the ambient space. We also derive upper and lower bounds on the expected dispersion for points chosen independently and uniformly at random from the Euclidean unit sphere. In terms of the corresponding inverse $\widetilde{N}(\varepsilon,d)$, our bounds are optimal with respect to the dependence on $\varepsilon$.
\medspace
\end{abstract}



\section{Introduction and main results}	

In this paper we study the minimal spherical dispersion of point sets on the Euclidean unit sphere $\mathbb{S}^{d}:= \big\{ x\in\mathbb{R}^{d+1}\,:\, \|x \|_2 = 1 \big\}$ with a focus on obtaining bounds that depend simultaneously on the dimension $d+1$ of the ambient space $\R^{d+1}$ and the number $n$ of points. The study of this quantity, which we shall define in a moment, was initiated by Rote and Tichy \cite{RT1995}, extending a concept previously introduced and investigated by Hlawka \cite{H1976} and Niederreiter \cite{N1983}. The motivation comes from typical problems arising in robotics, where one is interested in approximating general curves by simple ones (we refer to \cite{RT1995} for more information). There is also a significant body of research on the dispersion of the $d$-dimensional cube and torus, and so we just refer the reader to the recent work \cite{LL2021} and the references cited therein.

The Euclidean sphere $\SSS^{d}$ comes with a natural Borel probability measure $\pi_d$ given by the normalized Hausdorff measure. This measure is commonly known as the normalized spherical measure or normalized surface measure and corresponds to the uniform distribution on the sphere. A spherical cap $C(x,t)\subseteq \mathbb{S}^d$ with center $x\in \mathbb{S}^d$ and $t\in [-1,1]$ is given by
\[
C(x,t) := \Big\{y\in\mathbb{S}^d\,:\, \langle x,y \rangle >t \Big\},
\]
so it is simply an open half-space intersected with the sphere. The collection of spherical slices in $\R^{d+1}$, which we shall denote by $\mathcal{S}_d$, is given by intersecting two open half-spaces then intersected with the sphere. This can be expressed in terms of spherical cap intersections, more precisely, 
\[
\mathcal{S}_d := \Big\{ C(x,t)\cap C(y,s)\,:\,  x,y\in\mathbb{S}^d,\,s,t\in [-1,1]  \Big\}.
\]
As already mentioned, we are interested in the minimal spherical dispersion, i.e., the minimal dispersion with respect to the test set of spherical slices. To define this quantity, we first introduce the spherical dispersion of an $n$-element point set $\mathcal P_n:=\{x_1,\dots,x_n\}\subseteq \mathbb{S}^d$, which is given by
\[
	{\rm disp}\big(\mathcal P_n;d\big) := \sup_{B\in \mathcal{S}_d\atop{ \mathcal P_n \cap B =\emptyset}} \pi_d(B).
\]
This is the largest (in the sense of the normalized surface measure) spherical slice not containing any point of $\mathcal P_n$. For $d,n\in\N$, the \emph{minimal} spherical dispersion is defined as
\[
	{\rm disp}^{*}(n,d) := \inf_{\mathcal P\subseteq\SSS^d\atop{\#\mathcal P=n}} {\rm disp}\big(\mathcal P;d\big),
\]
i.e., the infimum of the dispersion over all possible point sets on $\SSS^d$ with cardinality $n$.

The work of Rote and Tichy already contains both lower and upper bounds on the spherical dispersion. However, the focus of their paper is on the dependence on the number $n$ of points, while we are aiming for simultaneous control in the number of points and the dimension of the ambient space; their upper bound is of the form $\mathcal{O}(d^3/n)$ and their lower bound of the form $\mathcal{O}(n^{-1})$.
We derive lower bounds depending not just on the number of points $n$, but also on the dimension $d$, and upper bounds improving upon the dimensional dependence $\mathcal{O}(d^3/n)$. More precisely, we shall prove the following results.

\begin{thmalpha}[Lower bounds]\label{thm:minimal spherical dispersion}
Let $d\in\N$. Then, for any $n \in\mathbb{N}$,
	\[
	{\rm disp}^*\big(n,d\big) \geq 
	\begin{cases}
	1/2 &: n\leq d+1 \\
	1/4 &: d+1<n\leq 2d+1 \\
	\frac{1}{2n-4d+4} &: 2d+1\leq n<3d-2\\
	\frac{1}{n-d+2} &: n\geq 3d-2.
	\end{cases}		
	\] 
\end{thmalpha}
\begin{figure}
	\begin{center}
		\includegraphics[height=8cm]{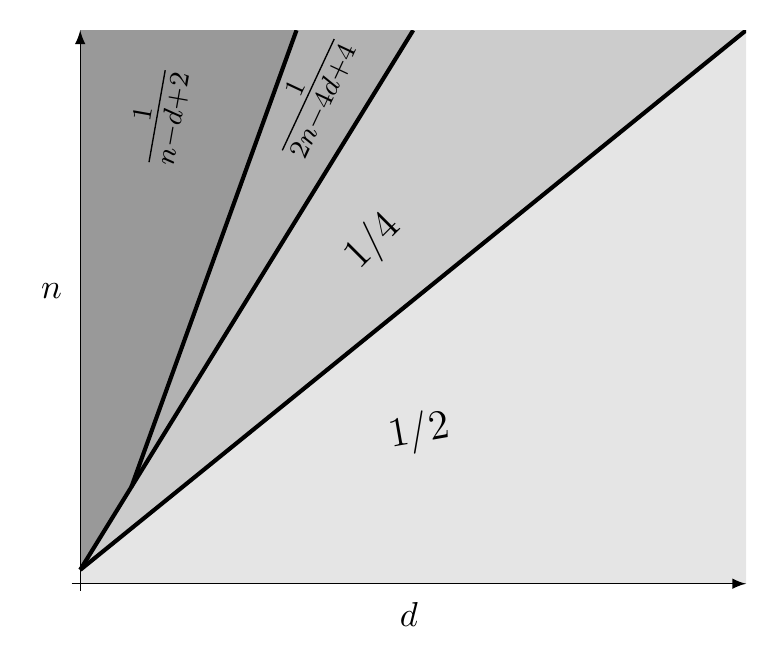}
		\label{fig}
		\caption{Visualization of lower bounds in Theorem~\ref{thm:minimal spherical dispersion} in dependence on the relation of $n$ and $d$.}
	\end{center}
\end{figure}

\begin{rmkalpha}\label{rem:inverse thm a}
It is instructive to reformulate Theorem \ref{thm:minimal spherical dispersion} in terms of the inverse of the minimal spherical disperison. This quantity is defined as 
\[
	N(\varepsilon,d) := \min\big\{n\in\mathbb{N}\,:\, {\rm disp}^{*}(n,d)<\varepsilon \big \}.
\]
We obtain from Theorem \ref{thm:minimal spherical dispersion} that, for any $\varepsilon\in(0,1/4)$,
\[
N(\varepsilon,d) \geq 
\max\left\{\frac{1}{\varepsilon}+d-2,\,\frac{1}{2\varepsilon}+2d-2\right\}.\]
It shows that for fixed $\varepsilon\in(0,1/4)$ the inverse minimal spherical dispersion grows linearly with respect to the dimension $d$.
\end{rmkalpha}

The proof relies on the following ideas. 
The lower bound is based on a test set expansion procedure and a suitable choice of appropriate hyperplanes that yield ``good'' spherical slices. For deriving our upper bound, we use the fact that the minimal dispersion is always smaller than the expected dispersion, in formulas,
\begin{equation} \label{eq: min_smaller_expec}
	{\rm disp}^*\big(n,d\big) \leq \E\big[{\rm disp}\big(\mathcal P_n;d\big)\big],
\end{equation}
where $\mathcal P_n := \{X_1,\dots,X_n \}$ and $X_1,\dots,X_n$ is an iid sequence of uniformly distributed random variables on the sphere $\SSS^d$. This approach is related to the recent work \cite{HKKR2019} (see comments after Theorem \ref{thm:expected spherical disp}), in which the authors study the expected dispersion of random point sets whose elements are uniformly distributed on the cube. Eventually, any upper bound on the expected dispersion leads to an upper bound of the minimal dispersion. The following theorem provides such an upper bound together with a lower bound and it is interesting to mention that the dependence on $n$ cannot be improved since there is a corresponding lower bound.  Our result reads as follows.

\begin{thmalpha}\label{thm:expected spherical disp}
Let $n,d\in\N$ with $n\geq 32d$. Assume that $X_1,\dots,X_n\in \SSS^d$ are independent random vectors chosen with respect to $\pi_d$. Then
\[
\frac{1}{9}\,\frac{\log(n)}{n}\leq \E\big[{\rm disp}\big(\mathcal P_n;d\big)\big] \leq \frac{64}{\log2} \frac{d}{n}  \log\left(	\frac{e n}{32 d}	\right),
\]  
where $\mathcal P_n := \{X_1,\dots,X_n \}$. 
\end{thmalpha}

In contrast to the result in \cite{HKKR2019}, where the upper bound on the expected dispersion is deduced from a $\delta$-cover approach (exploiting \cite[Lemma~1, Theorem~1]{R2018}), our upper bound on the expected spherical dispersion is in terms of the VC-dimension of the set $\mathscr S_d$ and uses a result of Blumer, Ehrenfeucht, Haussler, and Warmuth \cite{BEHW}. The lower bound in Theorem~\ref{thm:expected spherical disp} follows from similar arguments as in \cite{HKKR2019} adapted to our spherical framework. 

\begin{rmkalpha}\label{rem: inverse thm b}
We shall reformulate the bounds on the expected spherical dispersion from Theorem \ref{thm:expected spherical disp} in terms of its inverse, which is defined for all $\varepsilon\in(0,1)$ and $d\in\N$ as
\[
\widetilde{N}(\varepsilon,d):= \min\Big\{ n\in\N\,:\,\E\big[{\rm disp}\big(\mathcal P_n;d\big)\big] \leq \varepsilon \Big\}.
\]
From \eqref{eq: min_smaller_expec} we know that $N(\varepsilon,d)\leq \widetilde{N}(\varepsilon,d)$ such that, together with Remark~\ref{rem:inverse thm a}, we obtain 
for all $\varepsilon\in(0,\frac{1}{9e})$,
\[
\max\left\{\frac{1}{\varepsilon}+d-2,\,\frac{1}{2\varepsilon}+2d-2,\,\frac{1}{9\varepsilon}\log\Big(\frac{1}{9\varepsilon}\Big)\right\}
 \leq \widetilde{N}(\varepsilon,d) \leq 96 \frac{d}{\varepsilon}\log\Big( \frac{96}{\varepsilon} \Big),
\] 
i.e., the dependence on the parameter $\varepsilon$ and on $d$ is (individually) optimal.
\end{rmkalpha}

\begin{rmkalpha}
	Exploiting again the fact that $N(\varepsilon,d)\leq \widetilde{N}(\varepsilon,d)$ yields to an upper bound on $N(\varepsilon,d)$. We obtain
	from Remark~\ref{rem:inverse thm a} and Remark~\ref{rem: inverse thm b} that
	\[
	\max\left\{\frac{1}{\varepsilon}+d-2,\,\frac{1}{2\varepsilon}+2d-2\right\}
	\leq	N(\varepsilon,d)
	\leq 96 \frac{d}{\varepsilon}\log\Big( \frac{96}{\varepsilon} \Big)
	\]
	for $\varepsilon\in (0,\frac{1}{9\mathrm{e}})$. It shows that the inverse minimal spherical dispersion grows linearly with respect to $d$ for fixed $\varepsilon$ and almost linearly in $\varepsilon^{-1}$ (up to the $\varepsilon$-dependent logarithmic factor in the upper bound) for fixed dimension $d$. 
\end{rmkalpha}

\vskip 5mm

The rest of the paper is organized as follows. Section \ref{sec:proofs} is dedicated to the proofs and split into several parts, dealing first with the expected spherical dispersion as presented in Theorem \ref{thm:expected spherical disp} and then with the lower bound on the minimal spherical dispersion presented in Theorem \ref{thm:minimal spherical dispersion}.

%
%
%

\section{Proofs of the main results}\label{sec:proofs}

We shall now present the proofs of our main results and start with the bounds presented in Theorem \ref{thm:expected spherical disp}. After that we go over to arguing for the lower bounds on the minimal dispersion presented in Theorem \ref{thm:minimal spherical dispersion}.

\subsection{The expected spherical dispersion}
In \cite{HKKR2019}, Hinrichs, Krieg, Kunsch, and Rudolf studied the expected dispersion for iid random points in the $d$-dimensional cube $[0,1]^d$. 
We consider a generalized setting:
assume that there is a probability space $(B,\Sigma,\mu)$, where $B$ is equipped with a metric and $\Sigma$ is the corresponding $\sigma$-algebra of Borel sets. Let $\mathcal{B}\subseteq \Sigma$ be a family of subsets of $B$ which we call set of test sets. Then, for 
$\mathcal{P}=\{x_1,\dots,x_n\} \subseteq B$ define the $(\mathcal{B},\mu)$-dispersion of $\mathcal{P}$ as 
\[
{\rm disp}_\mu(\mathcal{P},\mathcal{B}) 
:= \sup_{T\in\mathcal{B},\, T\cap \mathcal{P} = \emptyset} \mu(T) = \sup_{T\in\mathcal{B}} \mu(T) \mathbf{1}_{0}(\vert T\cap \mathcal{P} \vert).
\]
We restrict ourselves to scenarios where $\mathcal{B}$ is countable, such that the supremum within the dispersion is taken over a countable set, which leads to the measurability of the mapping $x_1,\dots,x_n \mapsto {\rm disp}_\mu(x_1,\dots,x_n;\mathcal{B})$. 
Given a probability space $(\Omega,\mathcal{A},\mathbb{P})$, for $n\in\mathbb{N}$ let $X_1,\dots,X_n$ be an iid sequence of random variables, with $X_i \colon \Omega\to B$, where each $X_i$ is distributed according to $\mu$.
With this we define $\mathcal{P}_n := \{X_1,\dots,X_n\}$ and the expected dispersion as
\[
	\mathbb{E}[{\rm disp}_\mu(\mathcal{P}_n;\mathcal{B})]
\]
In contrast to the $\delta$-cover approach of \cite{HKKR2019}, our results on the upper bound of the expected dispersion are based on the VC-dimension. 
Moreover, we adapt the statement of the lower bound of the expected dispersion from \cite{HKKR2019} to our generalized situation. The proof follows as in \cite{HKKR2019}, but for the convenience of the reader we provide it. 
Eventually applying the former estimates leads to the upper and lower bound on the expected spherical dispersion for iid random points on the sphere $\SSS^{d}$. Finally, we provide the justification for the bounds of the inverse of the expected dispersion.

\subsubsection{The upper bound of the expected dispersion, proof of the upper bound of Theorem~\ref{thm:expected spherical disp}}

We start with defining the VC-dimension $d_\mathcal{B}$ of $\mathcal{B}$. It is the cardinality of the largest subset $P$ of $B$ such that the set system $\{P\cap T\colon T\in\mathcal{B} \}$ contains all subsets of $P$. Having this we are able to state an auxiliary result of Blumer, Ehrenfeucht, Haussler, and Warmuth. It follows by virtue of \cite[Lemma~A2.1, Lemma~A2.2 and Proposition~A2.1(iii)]{BEHW}.
\begin{lemma}
	For any $t>0$ and any $n\geq d_{\mathcal{B}}$, we have
	\begin{equation} \label{eq: conc_Blumer_et_al}
		\mathbb{P} \left( \rm{disp}_\mu(\{X_1,\dots,X_n\},\mathcal{B}) > t \right)
		\leq \left(\frac{en}{d_\mathcal{B}}\right)^{d_\mathcal{B}} 2^{-tn/2}.
	\end{equation}
\end{lemma}
With that estimate we provide the arguments for our upper bound of the expected dispersion.
\begin{proposition}
	\label{prop:exp_disp_bnd}
	Let $n\in\mathbb{N}$ with $n\geq d_{\mathcal{B}}$. Then
	\[
	\mathbb{E}[{\rm disp}_\mu(\mathcal{P}_n;\mathcal{B})] \leq \frac{4d_{\mathcal{B}}}{n} \log_2\left(\frac{e n}{d_\mathcal{B}}\right).
	\]
\end{proposition}
\begin{proof}
Set $\gamma:= \gamma_{d_\mathcal{B},n} = \min\left\{1, \frac{2d_{\mathcal{B}}}{n} \log_2 \left( \frac{en}{d_{\mathcal{B}}}\right)\right\}$. With this we have
\begin{align*}
	\mathbb{E}[{\rm disp}_\mu(\mathcal{P}_n;\mathcal{B})]
	& = \int_0^1 \mathbb{P} \left( \rm{disp}_\mu(\{X_1,\dots,X_n\},\mathcal{B}) > t \right) {\rm d}t\stackrel{\eqref{eq: conc_Blumer_et_al}}{\leq} \gamma +  \int_\gamma^1 \left(\frac{en}{d_{\mathcal{B}}}\right)^{d_{\mathcal{B}}} 2^{-tn/2} {\rm d}t\\
	& = \gamma + \left(\frac{en}{d_{\mathcal{B}}}\right)^{d_{\mathcal{B}}} \frac{2}{n \log 2} \left[ \left(\frac{en}{d_{\mathcal{B}}}\right)^{-d_{\mathcal{B}}} -2^{-n/2}\right]\\
	& \leq  \gamma + \frac{2}{n\log 2} \leq \frac{4d_{\mathcal{B}}}{n} \log_2 \left(\frac{en}{d_{\mathcal{B}}}\right).
	\qedhere
\end{align*}
\end{proof}
Observing that the mapping 
\begin{equation}
\label{eq: monotonicity}
x\mapsto \frac{4x}{n} \log_2\left( \frac{en}{x} \right)
\end{equation} 
is increasing for $x\in [1,n]$ leads to the fact that an upper bound of the VC-dimension gives an upper bound of the expected dispersion via Proposition~\ref{prop:exp_disp_bnd}. Therefore, we state a further tool for the application of the former proposition to the spherical dispersion. For a proof of the next result we refer to the application of \cite[Lemma~3.2.3.]{BEHW}.
\begin{lemma} \label{lem:upp_VC_dim_intersect}
	Given a family of test sets $\mathcal{B}$ with VC-dimension $d_{\mathcal{B}}$ we have that the VC-dimension of the new set of test sets
	\[
		\widetilde{\mathcal{B}} 
		:= \{ D\cap E\colon D,E\in\mathcal{B} \}
	\]
	satisfies $d_{\widetilde{\mathcal{B}}} \leq 4d_{\mathcal{B}} \log_2 6$.
\end{lemma}
In the setting of the spherical dispersion, we have $B=\mathbb{S}^d$, $\mu=\pi_d$, and $\mathcal{B}=\mathcal{S}_d$. 
Note that the test set of spherical slices $\mathcal{S}_d$ consists of intersections of spherical caps of $\mathbb{S}^d$. Therefore, by Lemma~\ref{lem:upp_VC_dim_intersect}, for obtaining an upper bound of $d_{\mathcal{S}_d}$, it is sufficient to provide the VC-dimension of the set of test sets of spherical caps, which for $d\in\mathbb{N}$ is denoted by
\[
	\mathcal{C}_d:= \big\{ C(x,t)\colon x\in\mathbb{S}^d, t\in [-1,1] \big\}.
\]
In the works \cite[Proposition~5.12]{bilyk2015random} and \cite[Proposition~42]{S2019} it is shown that $d_{\mathcal{C}_d} = d+2$. Here recall that $d$ is the ``classical dimension'', that is, $\mathbb{S}^d\subseteq\mathbb{R}^{d+1}$. Now we have all auxiliary results to state and prove our upper bound on the expected spherical dispersion.
\begin{proposition}
 Let $n,d \in\mathbb{N}$ with $n\geq 32 d$. Assume that $X_1,\dots,X_n$ are independent random points chosen uniformly distributed from $\mathbb{S}^d$ with respect to $\pi_d$. Then
 \[
 	\mathbb{E}\left[ {\rm disp}(\mathcal{P}_n;d) \right]
 	\leq \frac{64}{\log 2}\cdot  \frac{d}{n} \log \left(\frac{en}{32d}\right).
 \]
\end{proposition}
\begin{proof}
	The proof follows essentially as already indicated. From the VC-dimension of $\mathcal{C}_d$ mentioned before, we obtain $d_{\mathcal{S}_d} \leq 32 d$. Using the monotonicity from \eqref{eq: monotonicity} and considering only $n\geq 32 d$, leads by Proposition~\ref{prop:exp_disp_bnd} to the claimed upper bound of the expected spherical dispersion.
\end{proof}

\subsubsection{The lower bound of the expected dispersion
}

The proof of this result follows as in \cite{HKKR2019} and is based on the coupon collector's problem. We elaborate on some of the details for the sake of completeness and convenience. 
All random variables shall be defined on the common probability space $(\Omega,\mathcal A,\Pro)$. 
Let us start by recalling an elementary result from \cite[Lemma 2.3]{HKKR2019}, which follows from Chebychev's inequality.
\begin{lemma}\label{lem:coupon chebychev}
Let $\ell\in\N$ and $(Y_i)_{i\in\N}$ be a sequence of independent random variables uniformly distributed on the set $\{1,\dots,\ell\}$. If we set
\[
\tau_\ell := \min\big\{k\in\N\,:\,\{Y_1,\dots,Y_k\}=\{1,\dots,\ell\} \big\},
\] 
then, for any natural number $n\leq \big(\sum_{j=1}^\ell j^{-1}-2\big)\ell$, we have
\[
\Pro[\tau_\ell > n]>\frac{1}{2}.
\]
\end{lemma}
From this bound on the upper tail of $\tau_\ell$, we can deduce under a `decomposition condition' a lower bound on the expected $(\mathcal{B},\mu)$-dispersion for independent $\mu$-distributed points on $B$.
\begin{proposition}\label{prop:lower bound expected spherical disp}
Let $n,d\in\N$ and $\ell:= \lceil \frac{(1+e)n}{\log(n)} \rceil$, where $\lceil\cdot\rceil$ denotes the ceiling function, which maps $x\in\R$ to the least integer greater than or equal to $x$. Assume that there are pairwise disjoint test sets $S_1,\dots,S_{\ell} \in\mathcal{B}$ such that $B=\bigcup_{i=1}^{\ell} S_i$ and $\mu(S_j)=1/\ell$ for any $j=1,\dots,\ell$. Let $X_1,\dots,X_n$ be independent random points on $B$, where each $X_i$ is $\mu$-distributed. Then
\[
\E\big[{\rm disp}_\mu\big(\mathcal P_n;\mathcal{B}\big)\big] \geq \frac{1}{9}\,\frac{\log(n)}{n},
\]  
where $\mathcal P_n := \{X_1,\dots,X_n \}$.
\end{proposition}

\begin{proof}
We follow verbatim the proof in \cite{HKKR2019}, but include the details for the readers convenience. For $i\in\{1,\dots,n\}$, we define random variables $Y_i:\Omega\to\{1,\dots,\ell\}$ so that $Y_i(\omega)=j\in\{1,\dots,\ell\}$ if and only if $X_i(\omega)\in S_j$, i.e,
\[
Y_i := \sum_{j=1}^\ell j \mathbb 1_{\{X_i\in S_j\}}.
\]
This means that the $Y_i$'s are independent and uniformly distributed in $\{1,\dots,\ell\}$, and that the value of $Y_i$ indicates the test set from $S_1,\dots,S_\ell$ which the point $X_i$ falls in. Note that for any $\omega\in\Omega$ such that
\begin{equation}\label{eq:sets different}
\{Y_1(\omega),\dots,Y_n(\omega) \} \neq \{1,\dots,\ell \}
\end{equation}
(which means the left-hand side must be strictly contained in the right-hand side), there must exist at least one $r\in\{1,\dots,\ell \}$ such that $\{X_1(\omega),\dots,X_n(\omega) \} \cap S_r = \emptyset$. For those $\omega\in\Omega$ satisfying \eqref{eq:sets different}, we thus have
\[
{\rm disp}_\mu\big(\mathcal P_n(\omega);\mathcal{B}\big) \geq \frac{1}{\ell},
\]
where $\mathcal P_n(\omega):=\{X_1(\omega),\dots,X_n(\omega) \}$.
Therefore, on average, we obtain
\begin{equation}\label{eq:lower bound expectation}
\E\big[{\rm disp}_\mu\big(\mathcal P_n;\mathcal{B}\big)\big] = \int_{\Omega} {\rm disp}_\mu\big(\mathcal P_n(\omega);\mathcal{B}\big) \,\Pro(d\omega) \geq \frac{1}{\ell}\,\Pro\big[\{Y_1,\dots,Y_n \} \neq \{1,\dots,\ell \}\big] = \Pro[\tau_\ell>n],
\end{equation}
where $\tau_\ell$ is defined as in Lemma \ref{lem:coupon chebychev}. We now use that $\ell = \lceil \frac{(1+e)n}{\log(n)} \rceil $ and obtain
\[
\frac{n}{\ell} \leq \frac{\log(n)}{1+e} \leq \log\Big(\frac{(1+e)n}{\log(n)}\Big) - 2 \leq \log(\ell)-2 < \Big(\sum_{j=1}^\ell j^{-1}\Big) - 2,
\]
where it was used that $\sum_{j=1}^\ell j^{-1}>\log(\ell+1)$ and that for any $x>1$,
\[
\log\Big(\frac{(1+e)x}{\log(x)}\Big)-2-\frac{\log(x)}{1+e}\geq 0
\]
(where equality holds for $x=\exp(1+1/e)$). Therefore, $n\leq \big(\sum_{j=1}^\ell j^{-1}-2\big)\ell$, and so we can apply Lemma \ref{lem:coupon chebychev}, obtaining together with \eqref{eq:lower bound expectation} that
\[
\E\big[{\rm disp}_\mu\big(\mathcal P_n;\mathcal{B}\big)\big] \geq \Pro[\tau_\ell>n] >\frac{1}{2}.
\]
Altogether, this leads us to the estimate
\[
\E\big[{\rm disp}_\mu\big(\mathcal P_n;\mathcal{B}\big)\big] > \frac{1}{2\ell} \geq \frac{1}{2}\,\frac{\log(n)}{(1+e)n+\log(n)} > \frac{1}{9}\,\frac{\log(n)}{n},
\]
which completes the proof.
\end{proof}

\subsubsection{The inverse of the expected spherical dispersion -- Proof of Remark \ref{rem: inverse thm b}}

Let us consider the inverse of the minimal dispersion, which is, for every $\varepsilon\in(0,1)$ and $d\in\N$ as
\[
\widetilde{N}(\varepsilon,d):= \min\Big\{ n\in\N\,:\,\E\big[{\rm disp}\big(\mathcal P_n;d\big)\big] \leq \varepsilon \Big\}.
\]
Theorem \ref{thm:expected spherical disp} together with a simple computation (see \cite{HKKR2019}) show that whenever $\varepsilon\in (0,\frac{1}{9e})$, 
\[
	\widetilde{N}(\varepsilon,d) \geq \frac{1}{9\varepsilon}\log\Big(\frac{1}{9\varepsilon}\Big).
\]

For convenience of the reader we prove the following lemma which serves as tool to obtain the upper bound on the inverse of the expected spherical dispersion.
\begin{lemma}\label{lem:inverse dispersion bounds}
	Let $\varepsilon\in(0,1)$, $d\in\N$, and $c_1,c_2\in[1,\infty)$ be absolute constants. Define a differentiable function $f:(0,\infty)\to\R$ via $x\mapsto c_1 \frac{d}{x}\log\Big(c_2\frac{x}{d}\Big)$. 
	Then, for any real number $x\geq d\frac{a}{\varepsilon}\log\Big(\frac{a}{\varepsilon}\Big)$ with $a:=c_1(1+c_2/e)$,  we have  
	\[
	f(x) <\varepsilon.
	\] 
\end{lemma}
\begin{proof}
	First, we show that $f$ is decreasing for all $x>\frac{e d}{c_2}$ and, in particular, for any $x\geq a\frac{d}{\varepsilon}\log\Big(a\frac{\sqrt{d}}{\varepsilon}\Big)$. To see this, we first note that $f$ is differentiable and that, for all $x\in(0,\infty)$,
	\[
	f'(x) = c_1\frac{d}{x^2}\Bigg[1- \log\Big(c_2\frac{x}{d}\Big)\Bigg].
	\]
	The latter is obviously less than zero if and only if $1< \log\Big(c_2\frac{x}{d}\Big)$, which is equivalent to $x>\frac{ed}{c_2}$.
	
	For some $\beta\in(0,\infty)$, we define $g_\beta:(0,\infty)\to\R$ via $x\mapsto \frac{\log(\beta x)}{x}$. This function is again differentiable with $g_\beta'(x) = \frac{1}{x^2}(1-\log(\beta x))$ and satisfies
	\[
	g_\beta'(x)=0 \qquad\iff\qquad x=x_{\max}=\frac{e}{\beta}.
	\]
	In particular, $g_\beta(x_{\max})=\frac{\beta}{e}$. Hence, for each $x\in(0,\infty)$, $g_\beta(x) \leq g_\beta(x_{\max}) = \frac{\beta}{e}$.
	
	After this preparation, we shall now prove the desired estimate. For every $x\geq d\frac{a}{\varepsilon}\log\Big(\frac{a}{\varepsilon}\Big)$, due to the monotonicity of $f$,
	\begin{align*}
	f(x) & \leq \frac{c_1d}{a\frac{d}{\varepsilon}\log\Big(\frac{a}{\varepsilon}\Big)} \log\Bigg(c_2\frac{a}{\varepsilon} \log\Big(\frac{a}{\varepsilon}\Big)\Bigg) 
	= \frac{\varepsilon c_1}{a} \left[1+ \frac{\log\Bigg(c_2\log\Big(\frac{a}{\varepsilon}\Big)\Bigg)}{\log\Big(\frac{a}{\varepsilon}\Big)}\right] 
	= \frac{\varepsilon c_1}{a} \left[1+g_{c_2}\Big( \log\Big(\frac{a}{\varepsilon}\Big)\Big) \right],
	\end{align*}
	with the choice $\beta:=c_2$ above. Using the established bound on $g$, we obtain, for all $x\geq d\frac{a}{\varepsilon}\log\Big(\frac{a}{\varepsilon}\Big)$,
	\[
	f(x) \leq \frac{\varepsilon c_1}{a}[1+c_2/e ] \leq \varepsilon,
	\]
	where we used $a:=c_1(1+c_2/e)$ in the last estimate. 
\end{proof}

Using the upper bound of the expected dispersion from Theorem~\ref{thm:expected spherical disp} and eventually applying the previous lemma with $c_1=64/\log 2$, $c_2 = e/32$ 
leads to 
\[
\widetilde{N}(\varepsilon,d) \leq d \frac{a}{\varepsilon}\log\Big(\frac{a}{\varepsilon}\Big)
\]
with $a=\frac{64}{\log 2}(1+1/32)\leq 96$.

\subsection{Lower bound of the minimal spherical dispersion}

%

%
%

The general idea behind the lower bound exploits the following observation. Assume we are given a point set $\{ x_1,\dots, x_{d}\} \in \mathbb{S}^d$.
Then there exists a hyperplane $H$ containing the origin $0\in\R^{d+1}$ which supports all points, i.e., $\{x_1,\dots,x_{d}\}\in H$. Considering the half-space determined by this hyperplane intersected with the sphere $\mathbb{S}^d$ leads to a spherical slice (even a spherical cap) with $\pi_d$-volume of at least $1/2$. 

\subsubsection{Proof of Theorem~\ref{thm:minimal spherical dispersion}}

%

\begin{proof}
	Let $n\in\N$ and $\mathcal P_n:=\{x_1,\dots,x_n\}\subseteq \SSS^d$. For $n'\geq n$, consider the set $\mathcal P' :=\mathcal P_n \cup \{x_{n+1},\dots,x_{n'-1},x_{n'}\} $ with $x_{n+1},\dots,x_{n'}\in\SSS^d$. Then
	\begin{equation}
	\label{eq: low_bound}
	{\rm disp}\big(\mathcal P_n;d\big) \geq {\rm disp}\big(\mathcal P';d\big),
	\end{equation}
	because the number of points in the primed point set obtained by this procedure is non-decreasing.

	We now distinguish four regimes:
	\vskip 2mm
	\textbf{1st case: $n\leq d+1$:}
	Since $n\leq d+1$, we may consider an expanded point set $\mathcal P':= \mathcal P_n\cup\{x_{n+1},\dots,x_{d+1} \}$, where $\{x_{n+1},\dots, x_{d+1} \}\in\SSS^d$. Then there exists a vector $\theta\in\SSS^d$ and a corresponding hyperplane $H_\theta:= \{y\in \R^{d+1}\,:\, \langle \theta,y \rangle =0 \}$ perpendicular to $\theta$ such that $\{x_1,\dots,x_d,0\} \subseteq H_\theta$. The half-space $H_\theta^{+}:=\big\{y\in\R^{d+1}\,:\, \langle \theta, y\rangle >0 \big\}$ and the half-spaces $H_\theta^{-}:=\big\{y\in\R^{d+1}\,:\, \langle \theta, y\rangle <0 \big\}$ do not contain any points of $\{x_1,\dots,x_d\}$ and the same holds for the corresponding spherical caps $S^+:= H_\theta^{+} \cap \SSS^{d}$ and $S^-:= H_\theta^{-} \cap \SSS^{d}$. Moreover, 
	either $S^+$ or $S^-$ contains the point $x_{d+1}$, so that we set $S_0\in\{S^+,S^-\}$ to be the spherical cap which does not contain $x_{d+1}$. Eventually, by the continuity of measure, $\pi_d(S_0)=\frac{1}{2}$ such that the definition of the spherical dispersion thus yields
	\[
	{\rm disp}\big(\mathcal P';d\big) = \frac{1}{2}.
	\] 
	Using this together with \eqref{eq: low_bound}, we obtain
	\[
	{\rm disp}\big(\mathcal P_n;d\big) \geq {\rm disp}\big(\mathcal P';d\big) =\frac{1}{2}.
	\]
	\vskip 2mm
	
	\textbf{2nd case: $n\geq d+1$:} We proceed similar as in the previous case. Namely, we use the hyperplane spanned by the points $\{x_1,\dots,x_d,0\}$. In other words, there is a vector $\theta$ such that the aforementioned hyperplane is given by $H_\theta := \{y\in\mathbb{R}^{d+1}\colon \langle \theta,y \rangle =0\}$. As in the previous case we have the two half-spaces $H_\theta^+$ and $H_\theta^-$ as well as the corresponding spherical caps $S^+$ and $S^-$ that do not contain any point of $\{x_1,\dots,x_d\}$. Recall that $\pi_d(S^+)=\pi_d(S^-)=1/2$. Now by the pigeonhole principle, we can choose $S\in\{S^+,S^-\}$, such that the spherical cap $S$ contains at most $\lfloor (n-d)/2 \rfloor$ points from $\{ x_{d+1},\dots,x_n \}$. Further, we decompose $S$ in $\lfloor (n-d)/2 \rfloor+1$ spherical slices (given as intersections of two spherical caps) of equal $\pi_d$-measure. Formally, we just choose $\theta^\perp\in S$ with $\langle \theta,\theta^{\perp} \rangle = 0$ (i.e., $\theta^{\perp}\in \SSS^{d}\cap H_\theta$) and now define a rotation in the plane spanned by $\theta$ and $\theta^{\perp}$; such a rotation by an angle $\alpha$ is given by the orthogonal matrix $A:=\id_{(d+1)\times(d+1)} + (\cos(\alpha)-1)[\theta\otimes \theta + \theta^{\perp}\otimes \theta^{\perp}]$, where $\id_{(d+1)\times(d+1)}$ denotes the $(d+1)\times(d+1)$ identity matrix. Then we take $\alpha:= \pi/\lfloor \frac{n-d}{2}\rfloor$ and perform successive rotations of $H_\theta$ at this angle, considering spherical slices obtained from the intersection of the spherical cap $S$ with the spherical cap obtained from intersecting $\SSS^{d}$ with the rotation of $H_\theta$. 
	Then, again by the pigeon hole principle, we find at least one of the slices, say $S_0\in\mathcal{S}_d$, with $S_0\cap \mathcal{P}_n = \emptyset$, such that
	\[
		\pi(S_0) \geq \frac{1}{2(\lfloor (n-d)/2 \rfloor+1)} \geq \frac{1}{n-d+2}.
	\]
	Thus, by the definition of the spherical dispersion we have
	\[
		{\rm disp}\big(\mathcal P_n;d\big) \geq \pi(S_0) \geq \frac{1}{n-d+2}.
	\]
		
		\textbf{3rd case: $d+1<n\leq 2d+1$:}
	As in the 1st case we create an expanded point set $\mathcal P'$ by adding points until we have exactly $2d+1$ and then we take \eqref{eq: low_bound} into account. Without loss of generality let us assume that the first $2d$ points do not lie on a single hyperplane through the origin, because this would lead us eventually back to the 1st case. Let $H_1$ be the hyperplane in $\R^{d+1}$ determined by the points $\{0, x_1,\dots,x_{d} \}$ and let $H_2$ be the hyperplane determined by $\{0,x_{d+1},\dots,x_{2d}\}$. Assume that $\theta_1,\theta_2\in\SSS^d$ are the respective normals to these hyperplanes and choose those vectors such that $\langle \theta_1,\theta_2 \rangle \geq 0$;  in particular we know that $|\langle \theta_1,\theta_2 \rangle | \neq 1$. Let us look at the corresponding upper and lower open half-spaces, i.e., for $i=1,2$ define 
	\begin{align*}
	H^{+}_{\theta_i} & := \big\{y\in\R^{d+1}\,:\, \langle \theta_i,y \rangle > 0 \big\}, 
	\quad \text{and} \quad
	H^{-}_{\theta_i}  := \big\{y\in\R^{d+1}\,:\, \langle \theta_i,y \rangle < 0 \big\}.
	\end{align*}
	Define 
	\begin{align*}
			S^{+,+} & := H_{\theta_1}^+ \cap H_{\theta_2}^+ \cap \SSS^d, \qquad 
			S^{+,-} := H_{\theta_1}^+ \cap H_{\theta_2}^- \cap \SSS^d,\\  
			S^{+,-} & := H_{\theta_1}^- \cap H_{\theta_2}^+ \cap \SSS^d, \qquad 
			S^{+,-} := H_{\theta_1}^- \cap H_{\theta_2}^- \cap \SSS^d.
	\end{align*}
One can easily see that the spherical slices $S^{+,+},S^{+,-},S^{+,-},S^{+,-}$ are pairwise disjoint and that for any $S\in\{S^{+,+},S^{+,-},S^{+,-},S^{+,-}\}$ we have $S\cap\{x_1,\dots,x_{2d}\} = \emptyset$.
Moreover
	\[
	\pi_d\Big(S^{+,-} \cup S^{+,+} \Big) = \pi_d\Big(S^{-,-} \cup S^{-,+} \Big) = \frac{1}{2},
	\]
such that 
\begin{itemize}
	\item either $\pi_d(S^{+,-}) \geq 1/4$ or $\pi_d(S^{+,+}) \geq 1/4$; and
	\item either $\pi_d(S^{-,-}) \geq 1/4$ or $\pi_d(S^{-,+}) \geq 1/4$.
\end{itemize}	
Hence, there is an $S_1\in \{S^{+,-},S^{+,+} \}$ and an $S_2\in \{S^{-,-},S^{-,+}\}$ with  $\pi_d$-measure greater than or equal to $1/4$. Since $S_1$ and $S_2$ are disjoint, we can choose $S_0\in\{S_1,S_2\}$ in such a way that it does not contain $x_{2d+1}$, i.e., $S\cap\mathcal{P}'=\emptyset$.
Therefore
	\[
	{\rm disp}\big(\mathcal P';d\big)=\sup_{S\in\mathcal S_d\atop{S \cap \mathcal P' =\emptyset}}\pi_d(S) \geq \pi_d(S_0)\geq \frac{1}{4},
	\]
which together with \eqref{eq: low_bound} yields
	\[
	{\rm disp}\big(\mathcal P_n;d\big) \geq {\rm disp}\big(\mathcal P';d\big) \geq \frac{1}{4}.
	\] 
	
	\vskip 2mm
	\textbf{4th case: $n>2d$:}
	Let $H_1$ and $H_2$ be as in the 3rd case. This gives rise to four different spherical slices that do not contain any of the points $\{x_1,\dots,x_{2d}\}$, where two of the slices, say $S_1,S_2\in\mathcal{S}_d$, have $\pi_d$-measure of at least $1/4$. Thus, by the pigeonhole principle
	we find an $S\in \{S_1,S_2\}$ that contains at most $\lfloor(n-2d)/2 \rfloor$ points of $\{x_{d+1},\dots,x_n\}$. Now we decompose $S$ into $\lfloor (n-2d)/2 \rfloor +1 $ spherical slices of equal $\pi_d$-measure (given as intersections of two spherical caps) and get by the pigeonhole principle that there must exist a test set $S_{0}\in\mathcal{S}_d$ with $S_{0}\cap \mathcal{P}_n=\emptyset$ such that 
	\[
	\pi_d(S_0) \geq \frac{1}{4}\cdot\frac{1}{\lfloor (n-2d)/2 \rfloor +1 } \geq \frac{1}{2n-4d+4}.
	\]
	Therefore, we have 
	\[
	{\rm disp}\big(\mathcal P_n;d\big)=\sup_{S\in\mathcal S_d\atop{S \cap \mathcal P_n =\emptyset}}\pi_d(S) \geq \pi_d(S_0)\geq \frac{1}{2n-4d+4}.
	\]
		\vskip 2mm
	Eventually taking the maximum over the lower bounds derived in the previous four cases, depending on the range where those are satisfied, yields the claimed estimate.
\end{proof}

%
%
%
%
%
%

\subsection*{Acknnowledgement}
JP is supported by the Austrian Science Fund (FWF) Projects P32405 ``Asymptotic geometric analysis and applications'' and Project F5508-N26, which is part of the Special Research Program ``Quasi-Monte Carlo Methods:  Theory and Applications''. DR is partially supported by DFG project 389483880. We thank Simon Breneis for comments on a previous version of this paper.

\bibliographystyle{plain}
\bibliography{dispersion}

\bigskip
\bigskip
	
	\medskip
	
	\small

%
%
	
	\noindent \textsc{Joscha Prochno:} Faculty of Computer Science and Mathematics,
			University of Passau, Innstra{\ss}e 33, 94032 Passau, Germany
	
	\noindent
	{\it E-mail:} \texttt{joscha.prochno@uni-passau.de}

	\medskip
		
		\noindent \textsc{Daniel Rudolf:} Faculty of Computer Science and Mathematics,
		University of Passau, Innstra{\ss}e 33, 94032 Passau, Germany
		
		\noindent
		{\it E-mail:} \texttt{daniel.rudolf@uni-passau.de}

\end{document}